\documentclass[10pt, conference,letterpaper]{IEEEtran}

\usepackage{amsmath,amsfonts,amsthm,amssymb,bm, verbatim,dsfont,mathtools}
\usepackage{color,graphicx}
\usepackage{etoolbox}
\usepackage{tikz}
\usepackage{xr,xspace}
\usepackage{todonotes}
\usepackage{paralist}
\usepackage{caption,soul}
\usepackage{algorithm}
\usepackage{algorithmic}
\usepackage{framed}
 
 \usepackage[font=small,labelfont=bf]{caption}

\usepackage{subcaption}
\newtheorem{theorem}{Theorem}
\newtheorem{lemma}{Lemma}
\newtheorem{proposition}{Proposition}

\newtheorem{definition}{Definition}

\newtheorem{conjecture}{Conjecture}

\newtheorem{remark}{Remark}

\newtheorem{thm}{Theorem}[section]

\newtheorem{defn}[thm]{Definition}

\newcommand{\eps}{\epsilon}

\newcommand{\lam}{\lambda}

\newcommand{\ind}{\mathds{1}}

\newcommand{\be}{\begin{enumerate}}
\newcommand{\bi}{\begin{itemize}}
\newcommand{\ee}{\end{enumerate}}
\newcommand{\ei}{\end{itemize}}

\renewcommand{\P}{\mathbb{P}}

\newcommand{\sgn}{\mathrm{sgn}}

\newtheorem{example}[theorem]{Example}
\newcommand{\bexamp}{\begin{example}  \hrule \vspace{.1in} }
\newcommand{\eexamp}{\end{example}  \hrule  \vspace{0.1in} }

\usepackage[bookmarks, colorlinks=true, plainpages = false, citecolor = blue,linkcolor=red,urlcolor = blue, filecolor = blue]{hyperref}
\usepackage{amsmath,amsfonts,amsthm,amssymb,bm, verbatim,dsfont,mathtools}
\usepackage{color,graphicx}
\usepackage{etoolbox}
\usepackage{tikz}
\usepackage{xr,xspace}
\usepackage{todonotes}
\usepackage{paralist}
\usepackage{caption,soul}
\usepackage{algorithm}
\usepackage{algorithmic}
\usepackage{hyperref}

\renewcommand{\hat}{\widehat}
\renewcommand{\tilde}{\widetilde}
\newcommand{\reals}{{\mathbb{R}}}

\newcommand{\eexp}{{\rm e}}

\newcommand{\zero}{\mathbf 0}
\newcommand{\identityf}[1]{\mathbf 1_{\{#1\}}}

\newcommand{\calF}{{\cal F}}

\newcommand{\calH}{{\cal H}}

\newcommand{\calL}{{\cal L}}
\newcommand{\calT}{{\cal T}}

\renewcommand{\P}{\mathbb{P}}

\DeclareMathOperator{\Tr}{Tr}

\newcommand{\post}[2]{\begin{center} \includegraphics[width=#2cm]{#1} \end{center} }

\newtoggle{short}
\toggletrue{short}
\togglefalse{short}

\begin{document}

\title{On non-unique solutions in mean field games}
\date{\today}
\author{%
  \IEEEauthorblockN{Bruce Hajek and Michael Livesay}
  \IEEEauthorblockA{Department of Electrical and Computer Engineering
    and Coordinated Science Laboratory\\
    University of Illinois, 
    Urbana, IL 61801, USA\\
    Email: \{b-hajek, mlivesa2\}@illinois.edu}
}

\maketitle

\begin{abstract}
The theory of mean field games is a tool to understand noncooperative dynamic stochastic games with a large number of players.     Much of the theory has evolved under conditions ensuring uniqueness of the mean field game Nash equilibrium.   However, in some situations, typically involving symmetry breaking, non-uniqueness of solutions is an essential feature.    To investigate the nature of non-unique solutions, this paper focuses on the technically simple setting where players have one of two states, with continuous time dynamics, and
the game is symmetric in the players, and players are restricted to using Markov strategies.   All the mean field game Nash equilibria are identified for a symmetric follow the crowd game. Such equilibria correspond to symmetric  $\epsilon$-Nash Markov equilibria for $N$ players with $\epsilon$ converging to zero as $N$ goes to infinity.

In contrast to the mean field game,  there is a unique Nash equilibrium for finite $N.$   It is shown that fluid limits arising from the Nash equilibria for finite $N$ as $N$ goes to infinity are mean field game Nash equilibria, and evidence is given supporting the conjecture that such limits, among all mean field game Nash equilibria, are the ones that are stable fixed points of the mean field best response mapping.
\end{abstract}

\section{Introduction and related work}

The theory of mean field games was initiated independently by Huang, Caines, and Malham\'{e}
\cite{huang2007large} and Lasry and Lions \cite{LasryLions07}.
The setting of Huang et al. is linear quadratic Gaussian (LQG) control and the setting of
Lasry and Lions is continuous state Markov diffusion processes.   The work of
Gomes, Mohr, and Souza \cite{gomes2013continuous} translates much of the theory of
\cite{LasryLions07} into the context of continuous time finite state Markov processes.
The LQG and finite state settings are technically simpler than the setting of continuous
state Markov processes.   All three of these works impose assumptions implying uniqueness
of solutions to the mean field game equations.

\iftoggle{short}{}{
The notion of Markov perfect equilibrium was introduced in
\cite{MaskinTirole84}.   It is basically a Nash equilibrium in a controlled
Markovian dynamics  framework, such that each
player can use a strategy that selects control actions based on
the current states of all players.    In particular, the constraint on
strategies for Markov perfect equilibria rules out trigger strategies such
that some player can be punished for past actions.
Given a game and $\epsilon > 0$,  a strategy profile is
defined to be an $\epsilon$-equilibrium (or $\epsilon$ Nash equilibrium) if it is not possible
for any player to gain more  than $\epsilon$ in expected payoff by unilaterally deviating
from his/her strategy. 
}%

The paper \cite{huang2007large} establishes $\epsilon$-Nash equilibrium properties for strategy profiles
consisting of the decentralized individual control laws that result as responses to the  collective mass trajectory.
Condition H1 of \cite{huang2007large} is a key to guaranteeing uniqueness
of the mean field equations,  In particular, for the other parameters fixed, the
value of $r$ in the term for control cost,  $ru^2$, should not be too small.  In essence, condition H1 restricts the level of coupling
among the players.   The mean field game (MFG) equations are expressed as a fixed point of an
operator $\cal T$ in \cite{huang2007large}.   Proposition 4.5 of  \cite{huang2007large} states that the fixed point
for $\cal T$ is globally attracting under condition H1 in the paper.
Section VI of \cite{huang2007large} illustrates a cost gap between individual and global based controls.
This is an example of the fact that the social welfare at a Nash equilibrium in game theory does not need to
equal the maximum social welfare achievable if the players were to cooperate.

The paper \cite{gomes2013continuous} studies the continuous-time, finite state version of mean field game theory.
Assumption 3, p. 110,  gives a monotonicity condition that ensures uniqueness of solutions to the mean field game equations.
Proposition 4 of \cite{gomes2013continuous}, on the existence of a mean field game Nash equilibrium
is proved by using Brouwer's fixed point theorem applied to the map $\theta \mapsto \xi(\theta),$ which is analogous
to the map $\cal T$ of  \cite{huang2007large}.   The domain of $\xi$  is the set $\calF$ of uniformly Lipschitz continuous
functions on the interval $[0,T].$

In contrast, multiple solutions of the mean field equations naturally
arise in \cite{YinMehtaMeynShanbhag12},  where synchronization of coupled
oscillators requires solutions that depart from the incoherence solution.
The setup is similar to the discrete-state setting we consider
in that it is in continuous time,  the players
are coupled through their running costs, and players can take
actions depending on their own states and on the states of the other
players.   But the setup in  \cite{YinMehtaMeynShanbhag12} is
different in that  the state space is continuous -- specifically it is the unit circle, 
and the focus is on infinite horizon average cost.  
The running cost for
player $i$,  $c(\theta_i, \theta_{-i}) = \frac 1 n \sum_j (1/2)\sin^2((\theta_i - \theta_j)/2 ),$
is join the crowd type;  it is smaller if the states are closer together.  
It is similar to flocking of birds or synchronization of fireflies.   The separate
Brownian motions of different players tend to make
them drift apart, and it requires cost for them to try to stick together.
If the coefficient $R$ for the cost is large enough it is not worth the
players trying to stick close together, and for the MFG limit they will stay
uniformly distributed over the circle (i.e. the incoherence solution).
As $R$ crosses below some critical value $R_c,$ the incoherence
solution still exists but it becomes unstable and additional solutions
appear.    We find an equivalent phenomena for the simpler discrete
state model in this paper.   In addition, our setting is considerably
simpler than that of \cite{YinMehtaMeynShanbhag12}, allowing us
to examine the stability of the mean field map $\cal T$ for  a finite time horizon.

\iftoggle{short}{}{
{\em Some related papers with discrete state models}
The paper \cite{WeintraubBenkardVanRoy08} introduces the notion of
oblivious equilibrium and compares it to the stronger equilibrium notion of
Markov perfect equilibrium.    In a Markov perfect equilibrium,
the actions of any player can depend on the current states of all players.
In contrast, for an oblivious equilibrium, the actions of any player can depend
only on the state of the player itself.   This limits the abilities of players to
react to fluctuations in population dynamics for a finite number of players.
However, in the mean field limit, the population dynamics becomes deterministic,
in which case the difference between the two equilibrium concepts diminishes
in the large number of players limit.   That is the notion explored in
\cite{WeintraubBenkardVanRoy08}.   An approximation theory of
\cite{WeintraubBenkardVanRoy08} shows that an oblivious equilibrium under
certain technical conditions can
be approximated by a Markov perfect equilibrium, while the converse direction
is not necessarily true.  The setting of  \cite{WeintraubBenkardVanRoy08} is
discrete time throughout.

Papers \cite{AltmanAvrachenkov_etal08} and  \cite{TembineLeBoudec_etal09}
discuss MFG for discrete state Markov processes.
Paper \cite{TembineLeBoudec_etal09} considers a so-called Markov
decision evolutionary game.  It is similar to the classical evolutionary dynamics
setting, but in contrast to the classical setting, players have both a type (that doesn't
change) and an internal state (that evolves in a Markov fashion).  
The number of players involved in an event at a discrete time point
is stochastically bounded, so as the number of players converges to infinity,  time is
sped up and a continuous time limit results.  A mean field limit for fixed Markov
policies exists by a Kurtz type theorem.     The setting of \cite{AltmanAvrachenkov_etal08}
is also a discrete state Markov process for each player,   The models of both
\cite{AltmanAvrachenkov_etal08} and  \cite{TembineLeBoudec_etal09} assume the
players use so-called {\em stationary policies}, such that the action of a player depends on
the type of the player and internal state of the player, but not on the states of other players. 
Thus, the equilibrium concept is oblivious equilibrium. 
}

\section{Problem formulation}   \label{sec.construction}

The model we adopt is almost a special case of the model
of \cite{huang2007large}.   We consider $N+1$ players
with each having state space $\{0,1\}.$    The state $(i(t) : 0 \leq t \leq T)$ of a given player
evolves as a controlled Markov process with predictable control $\alpha_t$, such that the jump
probabilities of the state process are given
by
$$
P(i(t+h)=1-i | i(t)=i) = (\alpha_t  + \eta) h + o(h)
$$
for $h >0.$   The parameter $\eta \geq 0$ represents a background jump rate, so if $\eta > 0$ then the
process has minimum jump rate $\eta.$     The background jumping is similar in spirit to the
Brownian motions that work against coherence of the coupled oscillators in \cite{YinMehtaMeynShanbhag12}.
The objective function of the reference player is to select $(\alpha_t)$ to solve
$$
\min_{\alpha}  E\left[ \int_0^T  c(i(t) , \theta_t, \alpha_t) dt  + \psi(i(T), \theta_T)  \right] ,
$$
where $\theta_t$ is the fraction of other players in state 0 at time $t.$
The running costs are assumed to have the form
$c(i,\theta,\alpha)=f(i,\theta)+\frac{\alpha^2}{2},$
such that the  residence costs per unit time, $f(0,\theta)$ and $f(1,\theta)$,
and terminal costs, $\psi(0,\theta), \psi(1,\theta),$ are all bounded, and
uniformly Lipschitz continuous in $\theta.$

\paragraph{Hamilton Jacobi Bellman (HJB) equation for $N+1$ player system}
A state feedback control for a given player is a nonnegative function
$(\alpha(i,n,t))$ such that $i\in\{0,1\}$ represents the current state of the
player,  $n\in\{0,\ldots , N\}$ represents the number of other player
in state 0, and $t\in [0,T].$    Suppose the reference player uses a state
feedback control $(\alpha(i,n,t))$, and the other $N$ players use state feedback
control $(\beta(i,n,t)).$   Then $(i(t), n(t))_{0\leq t \leq T}$ forms a controlled
Markov process on  $\{0,1\}\times \{0,  1, \ldots , N \},$ where $i(t)$ represents the state
of the reference player and $n(t)$ represents the number of other
players in state 0.  The transition rates are as follows:\footnote{If $j\neq i$ then
$i$ itself is one of the ``other players" for player $j.$}
$$ {\small
\begin{array}{c|l}
\mbox{transition} & ~~~~~~~~~~~~~~\mbox{rate} \\ \hline
(i,n) \rightarrow (1-i,n)  &  \alpha(i,n,t) + \eta \\
 (i,n) \rightarrow (i,n+1)  & \gamma^+(i,n,t) \\
 & ~~=(N-n) (\beta(1, n+1-i,t) + \eta). \\
 (i,n) \rightarrow (i,n-1) & \gamma^-(i,n,t) \\
& ~~=  n (\beta(0, n-i,t) + \eta).
\end{array} }
$$
Denote the cost-to-go function for the reference player by $u(i,n,t).$ 
The HJB equations for it are:
\begin{align}
&-\dot{u}(i,n,t)   =  f(i,n) -  \frac{((\alpha^*(i,n,t))^2}{2}\nonumber   \\
&~~~~ + \eta (u(1-i,n,t) - u(i,n,t))  \nonumber \\
& ~~~~ +  \gamma^+(i,n,t)(u(i,n+1,t)-u(i,n,t))  \nonumber  \\
&~~~~+ \gamma^- (i,n,t)(u(i,n-1,t)-u(i,n,t)),  \label{eq:HJB_N+1} \\
& u(i,n,T)=\psi(i,n)   \label{eq:HJB_N+1_bc}
\end{align}
where the corresponding control policy is 
\begin{align}
\alpha^*(i,n,t)= ( u(i,n,t)-u(1-i,n,t) )_+  \label{eq:HJB_N+1_policy}.
\end{align}
The HJB equations \eqref{eq:HJB_N+1}-\eqref{eq:HJB_N+1_policy} can be viewed in two different ways.
\begin{itemize}
\item For policy $\beta$ of the other $N$ players fixed, \eqref{eq:HJB_N+1} - \eqref{eq:HJB_N+1_policy}
determine the best response policy for the reference player.   i.e. $\alpha^*=BR(\beta).$
\item To find a symmetric Nash equilibrium,  replace $\alpha(\cdot, \cdot, t)$ and $\beta(\cdot, \cdot, t)$
by  $\alpha^*(\cdot,\cdot, t)$ in the definition of $\gamma^{\pm}$ and \eqref{eq:HJB_N+1}- \eqref{eq:HJB_N+1_policy}.
This yields a $2(N+1)$ dimensional ode with terminal boundary condition and Lipschitz continuous right hand
side that uniquely determines the functions $(u(i,n,t))$ and, hence also, the feedback control law $\alpha^*.$
The strategy profile such that all $N+1$ players use $\alpha^*$ is a Markov perfect Nash equilibrium, because
$\alpha^*$ is determined backwards from the terminal condition yielding a best response for any interval of the
form $[t,T].$   Moreover, the Markov perfect equilibrium is the unique
Nash equilibrium among all Markov type (i.e. state feedback) strategy profiles, because the similar HJB equations
for a more detailed model description with state space $\{0,1\}^{N+1}$ still has a unique solution and it is necessarily
invariant under permutation of the players.
\end{itemize}


\paragraph{Mean field game equilibria and map}

A mean field game Nash equilibrium for the finite horizon problem with
initial value $\overline{\theta}$ is any solution $(\theta_t,  u(i,t))$ to the following equations.\footnote{Note
the double use of notation ``$u.$"   We  write $u(i,t)$ for $u$ associated with mean field game solutions and $u(i,n,t)$
for $u$ associated with the $N+1$ player Markov perfect equilibrium.}
\begin{align}
&\dot{\theta_t} =(1-\theta_t)( (u(1,t)-u(0,t))_+ + \eta) \nonumber \\
&  ~~~~~ -  \theta_t  ((u(0,t)-u(1,t))_+ + \eta )    \label{eq:MFG_1} \\
&-\dot{u}(i,t)= f(i,\theta_t,t)  - \eta  (u(i,t)-u(1-i,t))       \nonumber \\
&~~~ - \frac{  ( (u(i,t)-u(1-i,t))_+ )^2}{2}    \label{eq:MFG_2}  \\
&   \theta_0= \overline{\theta}, \quad  u(i,T)= \psi(i,\theta_T).    \label{eq:MFG_3}
\end{align}
Note that the boundary conditions \eqref{eq:MFG_3} include both initial and terminal values.
The mean field equations \eqref{eq:MFG_1}-\eqref{eq:MFG_3} can be written
as a fixed point equation, $\theta = \calT(\theta)$,  where  ${\cal T}$ maps a collective mass trajectory $(\theta_t : 0 \leq t \leq T)$ to another
trajectory.    It is determined by first computing the decentralized individual control laws for the players.   Then by the
uniform law of large numbers
\iftoggle{short}{\cite{GineZinn}}{(see Appendix \ref{sec:unif_law_large_numbers})}, if each of the players follows the
same decentralized individual control law, their state processes  will be independent and the
empirical average of such processes will converge to an expected $\tilde \theta$ that is the output collective
mass trajectory.   More concretely, $\calT(\theta)$ is defined as follows.
First, cost-to-go functions $(u(i,t))$ are determined by the HJB terminal value problem
for a single player, in response to the collective mass trajectory $\theta.$
\begin{align}
 -\dot{u}(i,t) = f(i,\theta_t) - \frac{  ( (u(i,t)-u(1-i,t))_+ )^2}{2}\nonumber  \\- \eta  (u(i,t)-u(1-i,t))  \label{eq:HJB_one_player}  \\
  u(i,T) = \psi(i,\theta_T).      ~~~~\mbox{boundary condition at $T$}  \label{eq:HJB_one_player_init}
\end{align}
Then $\tilde \theta_t,$ the probability a single player using the decentralized state-feedback
control  $\alpha_t(i,t) = (u(i,t)- u(1-i,t))_+$   is in state $0$ at time $t$,  
is determined by the initial value problem (Kolmogorov forward equation):
\begin{align*}
&\dot{\tilde \theta_t} = (1-\tilde \theta_t)( (u(1,t)-u(0,t))_+ + \eta)   \\
&~~~~~~~~   -  \tilde \theta_t  ((u(0,t)-u(1,t))_+ + \eta )   \\
&  \tilde \theta_0 = \overline{\theta}  ~~~~~~~~~~~~~~\mbox{boundary condition at 0} 
\end{align*}
Motivated by the law of large numbers,  $\tilde \theta$ is defined to be the new collective mass trajectory,
i.e.  $\tilde\theta = \calT(\theta).$

The mean field game equations \eqref{eq:MFG_1} and \eqref{eq:MFG_2}, with the addition of an average
cost per unit time term $\kappa$ on the right-hand side of \eqref{eq:MFG_2} correspond to an {\em infinite
horizon game} for average cost per unit time.    (See  \cite{gomes2013continuous}, Section 2.12, p. 117.)
In that case the value functions $u(i,t)$ represent realative  cost to go.   The boundary conditions \eqref{eq:MFG_3}
are replaced by the condition that $\theta$ be constant in time or be periodic.   

\paragraph{Fluid limits of Markov perfect equilibrium}

As noted in the introduction, there can be multiple mean field game Nash equilibria, even for a finite
horizon problem with given  boundary conditions.
A mean field game Nash equilibrium $(\theta_t, u(i,t))$
yields a decentralized player strategy  $\alpha_t(i,t) = (u(i,t)- u(1-i,t))_+.$    For finite $N$, the strategy profile
such that every player uses $(\alpha_t(i,t))$ is easily seen to be an $\epsilon$-Nash equilibria such that
$\epsilon \to 0$ as $N \to\infty.$
\iftoggle{short}{For details see the appendix of the full version of this paper.\footnote{See full version at arXiv.org.}}
{(See Appendix \ref{sec:unif_law_large_numbers}.)}

 However, for finite $N$ there is a unique Markov perfect Nash equilibrium strategy profile,
 so for a given initial condition, the distribution of the finite $N$ system is uniquely determined.
 It is natural, therefore, to single out collective mass trajectories that arise as limits of the mass trajectories
 for Markov perfect equilibria.
 
 \begin{defn} \label{def:fluid_traj}
 Let $n^N(t)$ denote the number of players in state 0 at time $t$ under the unique symmetric Markov perfect
 equilibrium for the $N+1$ player game, and for some initial condition depending on $N.$
Then $\theta = (\theta_t : 0 \leq t \leq T)$ is a {\em fluid limit Markov perfect trajectory} (FLMP trajectory) if for some
sequence of initial states  with  $\lim_{N \to \infty} \frac{n^N(0)}N \to \theta_0,$  the following holds for any $\eps>0$, 
\begin{align} \label{eq:fluid_lim}
\lim_{N \to \infty}  \P\left[   \bigg|  \frac{n^N(t)}{N+1} - \theta_t \bigg|  < \eps \mbox{ for } 0 \leq t \leq T  \right] =1. 
\end{align}
\end{defn}

\begin{proposition}  \label{prop:FLMP_MFG}
Suppose $\eta >0.$   An FLMP trajectory is a mean field game Nash equilibrium.
\end{proposition}

See Appendix   \ref{FLMP_MFG_proof} for a proof.
We conjecture the proposition is also true for $\eta =0,$ but a change
of probability measure argument in the proof breaks down if $\eta = 0.$
Proposition \ref{prop:FLMP_MFG} raises the question of how to identify
which mean field Nash equilibria are FLMP trajectories.

\paragraph{Contributions of the paper}

Proposition  \ref{prop:FLMP_MFG} is new and its proof extends to the general setting
of \cite{gomes2013continuous}.   It shows that the search for FLMP trajectories can be
limited to the mean field game Nash equilibria.
The next contribution of this paper is to identify all of the MFG equilibria for a natural special case of the two state
model called follow the crowd.   This model is analogous to the model of synchronization of oscillators game
\cite{YinMehtaMeynShanbhag12},  but considerably simpler, so we can identify the finite horizon
solutions as well as the infinite horizon ones.
The third contribution is to offer the following conjecture, and give evidence for it:
\begin{conjecture}   \label{conj:fluid_trajectories}
The FLMP trajectories are the stable fixed points of the MFG mapping $\calT.$
\end{conjecture}
A similar type of conjecture is implicit in  \cite{YinMehtaMeynShanbhag12} based
on a notion of stability for constant, long-term average cost infinite horizon solutions,
called linear asymptotic stability.  The paper  \cite{YinMehtaMeynShanbhag12}
identifies the critical cost threshold at which the incoherence solution becomes unstable.
 In addition to giving evidence for Conjecture \ref{conj:fluid_trajectories} in the setting of finite horizon games,
we also show that the results of \cite{YinMehtaMeynShanbhag12} for constant, long-term average cost
infinite horizon solutions,  carry over to the setting of two state Markov processes.   For the infinite horizon framework,
we show asymptotic stability of certain fixed points for the nonlinear dynamics in Section \ref{sec:infinite_horizon_transient},
\iftoggle{short}{and an appendix in the full version of this paper}{and Appendix \ref{app:linear asymptotic stability}}
gives an analysis based on the notion of linear asymptotic stability introduced in \cite{YinMehtaMeynShanbhag12}.
Additional results are given in the appendix of
\iftoggle{short}{the full version of}{} this paper, including, for contrast, a similar analysis for
an avoid the crowd model with unique mean field game solutions,  and a description of a partial differential
equation (PDE) (given for more general model in \cite{gomes2013continuous})
 that can be considered to be an extension of the notion of mean field game.

\section{MFG equilibria for follow the crowd}

The {\em follow the crowd} model corresponds to the following
cost per time spent in state $i$:
\begin{align*}
f(i,\theta)=|1-\theta - i | =\left\{ \begin{array}{cl} 1-\theta & i=0 \\
\theta & i=1
   \\
\end{array} \right.
\end{align*}
In particular, if  $\theta > 1/2$ (more than half of the other players in state 0),
then state 0 has smaller cost per unit time than state 1.

Letting  $y = u_1-u_0,$ $x = 2\theta -1,$ the mean field equations
 \eqref{eq:MFG_1}- \eqref{eq:MFG_3}  can be written as:
\begin{align}   \label{eq:follow_MFGxy}
\begin{array}{l}
~~ \dot x   = y-x|y| - 2\eta x  \\
-\dot y = x - \frac12 y|y| - 2\eta y   
\end{array}
\end{align}
with the boundary conditions  $x_0 = 2\overline{\theta} - 1$ and
$y_T=\psi\left(1,\frac{1+x_T} 2\right) - \psi\left(0,\frac{1+x_T} 2\right).$
Once a solution $(x,y)$ to \eqref{eq:follow_MFGxy} is found
for the finite horizon problem over $[0,T]$, a corresponding solution
$(u_0,u_1,\theta)$ to the mean field game equations
can be found by simply integrating \eqref{eq:MFG_1}- \eqref{eq:MFG_2}
 because the righthand sides of  \eqref{eq:MFG_1}- \eqref{eq:MFG_2}
 are determined by $(x_t, y_t).$

 A useful fact is that the equations \eqref{eq:follow_MFGxy}  form a Hamiltonian system,
for the Hamiltonian function $H$:
\begin{align} \label{eq:Hxy}
H(x,y)= \frac{x^2 - 4 \eta xy + y^2 - xy|y|}{2}.
\end{align}
In other words, \eqref{eq:follow_MFGxy} has the form  $\dot x = H_y$ and $\dot y = -H_x,$
where $H_x$ and $H_y$ represent partial derivatives of $H.$
Consequently, the value of $H$ is constant along the solutions of  \eqref{eq:follow_MFGxy},
because $\frac{dH(x_t,y_t)}{dt} = \langle \nabla H,    \binom{H_y}{-H_x}\rangle \equiv 0,$  so
the trajectories trace out level contours of $H.$  This model is a special case
of potential mean field games defined in \cite{gomes2013continuous}, Section 5, for which Hamiltonians
exist. 

Contour maps of $H$ are shown in  Fig.  \ref{fig:follow_beta_Ham} for various values of $\eta.$
\begin{figure}[htbp]
\post{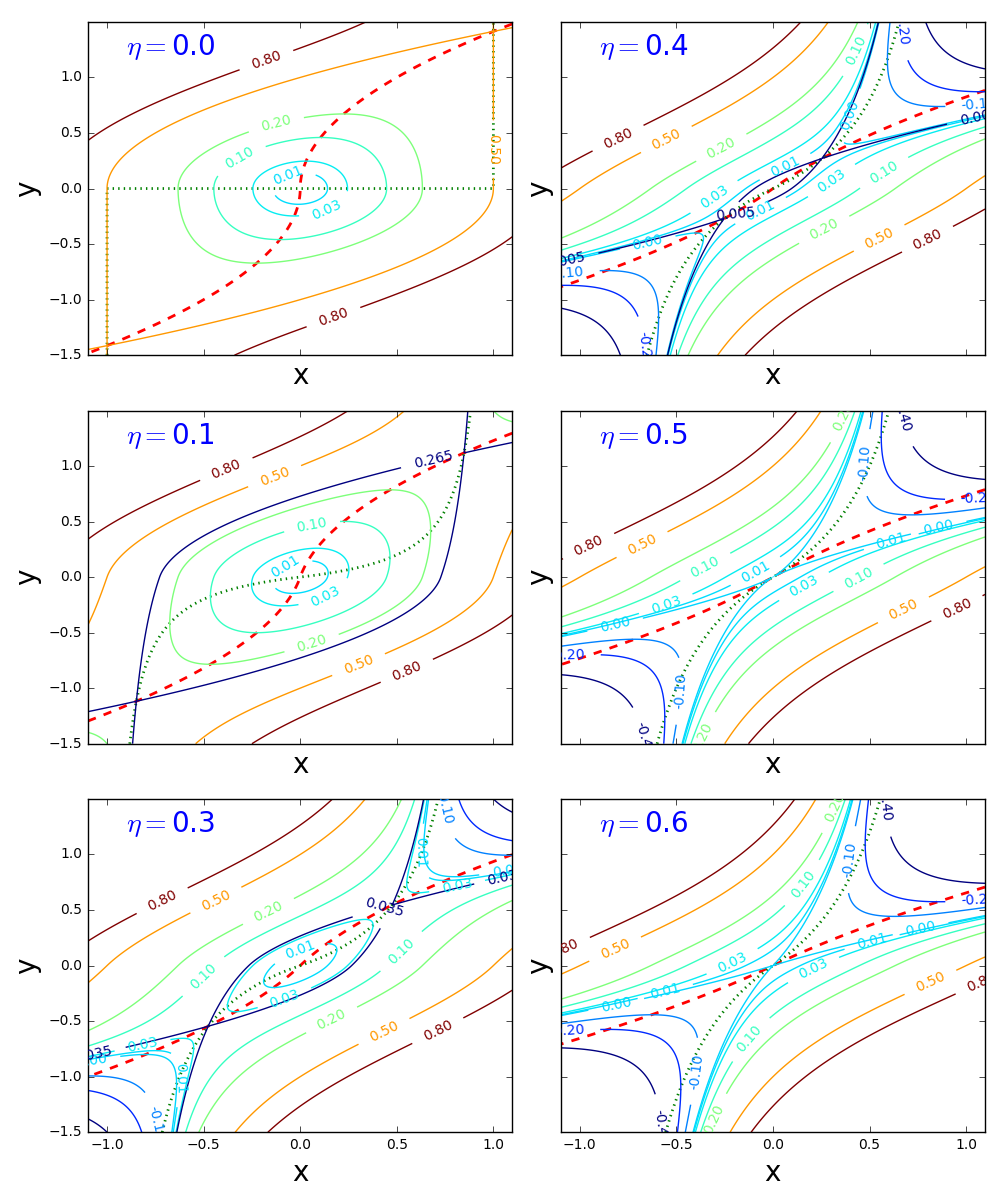}{8}
\caption{Contour plot of $H$ for several values of $\eta.$   Dashed lines are the zero sets
of  $H_x,$ and dotted lines are the zero sets of $H_y.$  The intersections
of dotted and dashed lines are the critical points of $H$  (i.e. solutions to $\nabla H = 0.$)}
\label{fig:follow_beta_Ham}
\end{figure}
For small values of $x,y$ the quadratic terms in
$ H$ dominate the cubic term, and for $\eta < 1/2$, constant
$x^2 - 4 \eta xy + y^2$ gives elliptical orbits of $x,y,$ in the
clockwise direction.

\subsection{Finite time horizon mean MFG solutions}
For the finite horizon mean field game with zero terminal cost (i.e. terminal boundary condition $y_T=0$),
and initial state $x_0=0$, correspond
to paths that begin on the $y$ axis (so the initial condition $x_0=0$ is satisfied) and end on the
$x$ axis.    One solution is $(x_t,y_t) \equiv (0,0)$ for $0\leq t \leq T.$ 
Let $\phi = \arctan\left( \frac x y \right)$ denote the angle of $x,y$ from the positive $x$ axis.
The angular velocity of $(x,y)$ is given by
\begin{align}  \label{eq:angular_xy}
\dot \phi = \frac{\dot y x - y \dot x }{x^2+y^2} =  -1  + \frac{\frac 3 2 xy|y| + 4\eta xy}{x^2 + y^2}
\end{align}
It is negative along the $y$ axis, indicating clockwise motion.  If $\eta \geq 1/2$ then $\dot\phi > 0$ along
the line $x=y$, indicating that $y=0$ is never reached.   Thus, if $\eta \geq 1/2$,  the trajectory
$(0,0)$ is the only MFG equilibrium.

If $\eta < 1/2$ then $\dot\phi < 0$ for $(x,y)$ in a neighborhood of
the origin, indicating clockwise movement.  Moreover, for $\phi$ fixed,  $\dot\phi$ is an increasing
function of the distance of $(x,y)$ to the origin (decreasing angular speed because angular velocity is negative).
Thus, the time for $(x,y)$ to traverse a contour across
the first quadrant is increasing in $y_0.$ for $y_0 > 0.$    As $y_0 \to \zero$ the dynamics is given,  to first order,
by the MFG linearized about $(0,0)$, given by
\begin{align}   \label{eq:follow_MFGxy_linearized}
\begin{array}{l}
~~ \dot x   = y - 2\eta x  \\
-\dot y = x  - 2\eta y   
\end{array}
\end{align}
with solution of the form (setting $x_0=0$ and $y_0>0$):
\begin{align*}
x_t &= \sin\left(\sqrt{1-4\eta^2}~t \right)  \\
y_t &= \sin\left(\sqrt{1-4\eta^2}~t  + \arccos(2\eta) \right)
\end{align*}
The time it takes the linear system to traverse the first quadrant is $T_c(\eta) \triangleq  \frac{\pi - \arccos(2\eta)}{\sqrt{1-4\eta^2}}.$
Hence, as $y_0 \to 0,$  the traversal time for the quadrant converges to $T_c(\eta).$
Thus, for $\eta < 1/2$ and $T \leq T_c(\eta),$   $(0,0)$ is the unique solution to the
MFG.    For  $T  >  T_c(\eta)$ there is one more solution that remains in, and traverses, the first quadrant,
and the negative of that solution remains in, and traverses, the third quadrant.
For $T$ large enough there are solutions that traverse contours of $H$ through three quadrants,  five quadrants, and so on.
A similar radial velocity analysis for the pair $(y,\dot y)$ (see
\iftoggle{short}{appendix of full version of the paper)}{Appendix \ref{sec:period_monotonicity})}
establishes that the entire periods of the dynamical system are increasing with amplitude,
as illustrated in  Fig. \ref{fig:multi_solutions}.   Since the dynamics is symmetric under rotation by $\pi,$ we
conclude that for any odd number $k$, starting on the positive $y$ axis, the time required to rotate through $k$ quadrants
is increasing in the initial condition $y_0.$   Therefore,  as $T$ increases from 0, the number of solutions starts at one
and jumps up by two when $T$ crosses times of the form
$T_c + k\pi /(\sqrt{1-4\eta^2})$ for $k\geq 1.$  Equivalently, the number
of solutions is
$1 + 2  \left\lceil  \frac{(T-T_c)\sqrt{1-4\eta^2}}{\pi}  \right\rceil.$
\begin{figure}\centering
\post{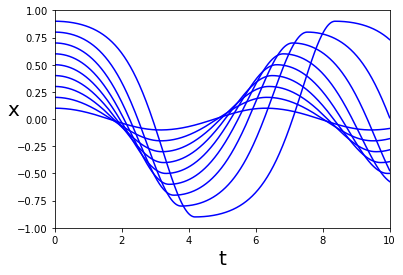}{4}
\caption{Several solutions with various terminal values of $x$ run backwards in time, for
follow the crowd dynamics with $\eta=0.$}
\label{fig:multi_solutions} 
\end{figure}

\subsection{Infinite horizon constant or periodic MFG solutions}
The equilibrium points of the dynamics \eqref{eq:follow_MFGxy}
are the critical points of the Hamiltonian function (i.e.  $\nabla H =0$),
and are given as follows. 
If $0 \leq \eta < 0.5,$ $(0,0)$ is an equilibrium point and 
there are also exactly two nonzero equilibrium points, given by $\pm \overline{P}$, where
\begin{align}   \label{eq:resting_points}
\overline{P} = \binom{\overline{x}}{\overline{y}} \triangleq  \binom{1-\eta^2 - \eta\sqrt{2+\eta^2}}{\sqrt{2+\eta^2} - 3 \eta}.
\end{align}
If $\eta \geq 0.5,$ $(0,0)$ is the unique equilibrium point.

Regarding infinite horizon periodic solutions,
examination of $H$ and the equations for angular velocity, \eqref{eq:angular_xy} and similar equation for angle of $(y,\dot y)$,
lead to the following conclusions.    If $0 \leq \eta < 0.5,$  
there is a two-dimensional family of periodic solutions that can be indexed by the peak amplitude of $x$ (ranges over $(0,\overline{x}))$
and phase.  The period of the solutions increases continuously
over $(2\pi/\sqrt{1-4\eta^2}, \infty)$ as the peak amplitude of $x$ increases over $(0,\overline{x})$.
 If $\eta \geq 0.5,$   there are no periodic solutions of \eqref{eq:follow_MFGxy}.

 \subsection{Infinite horizon convergent transient MFG solutions, and the asymptotically stable
constant solutions}   \label{sec:infinite_horizon_transient}
 
Consider the initial value problem over $t \in [0, \infty)$ with some initial condition  $(x_0,y_0)$  and
dynamics \eqref{eq:follow_MFGxy}.   First, suppose $0 \leq \eta < 0.5.$  For any initial condition $(x_0,y_0)$
such that $x_0 \neq 0$,  one of four cases holds:  $x_t$ is periodic with
a positive period,  $x$ converges to $\overline P$, $x$ converges to $-\overline P$, 
or $x_t$ exits $[-1,1]$ in finite time.   The following categorize the convergent solutions such that $x_t$ remains in $[-1,1].$
\begin{itemize}
\item For any initial value
of $x_0 \in (-\overline{x}, \overline{x})$,  there exist two corresponding initial values of $y_0$ such that
the solution of the initial value problem satisfies  (i) $x_t \in [-1,1]$ for all $t$ and  (ii) the solution converges to a
limit as $t \to \infty.$    For the smaller value of $y_0$ the limit is $-\overline{P}$ and for the larger value of $y_0$ the limit is $\overline{P}.$
The value of the larger $y_0$ for example is such that the contour of $H$ through $(x_0,y_0)$ contains $\overline{P}.$
\item For an initial value $x_0 \in [-1,-\overline{x}]$ there exists a unique value of $y_0$ such 
that the solution of the initial value problem satisfies $x_t \in [-1,1]$ for all $t.$   That solution converges  to
$-\overline{P}$ as $t\to\infty.$  
\item Similarly,  for an initial value $x_0 \in [\overline{x},1]$ there exists
a unique value of $y_0$ such  that the solution of the initial value problem satisfies $x_t \in [-1,1]$ for all $t.$
That solution converges  to  $\overline{P}$ as $t\to\infty.$ 
\end{itemize}

Second, suppose $\eta \geq 0.5.$  
For any $x_0 \in [-1,1]$, there is a unique
value of $y_0$, such that the solution of the initial value problem for \eqref{eq:follow_MFGxy}
satisfies $x_t \in [-1,1]$ for all $t.$   Furthermore,  $y_0$ has the same sign as $x_0$, and the
solution converges to $(0,0)$ as $t \to \infty.$  The value of $y_0$ is the root of
$H(x_0,y_0)=0$ (for $x_0$ fixed) that is closer to zero.

 The above observations give a sense in which $\pm \overline P$ is an asymptotically stable
 equilibrium point of the dynamics \eqref{eq:follow_MFGxy} if $0 \leq \eta < 0.5,$  and
 $(0,0)$ is an asymptotically stable equilibrium point if $\eta \geq 1/2.$   This sense
 of stability is not the usual definition of (Lyapunov) stability because we ask, for given $x_0$,
 whether there exists an associated value of $y_0$ giving the desired convergence.   The asymptotically
 stable limit points are saddlepoints of $H.$
 
As mentioned above, a related definition of stability, called linear asymptotic stability, is formulated in \cite{YinMehtaMeynShanbhag12}.
That definition and the results of \cite{YinMehtaMeynShanbhag12} for it are translated to the model of this
paper in \iftoggle{short}{the appendix in the full version of this paper.}{Appendix \ref{app:linear asymptotic stability}.}

\section{Evidence for  Conjecture \ref{conj:fluid_trajectories}}

In order to explore whether  Conjecture \ref{conj:fluid_trajectories} is true,
it is natural to explore two sides of the question.   One side is to identify the FLMP
trajectories.    Numerically that can be done by solving the $2(N+1)$
dimensional HJB equation for the system with $N+1$ players to find
the strategy $\alpha^*(i,n,t)$ players use for the Markov perfect equilibrium with $N+1$
players, and then either simulating the corresponding occupancy process through
Monte Carlo simulation of $N+1$ players independently using that policy,
or solving the Kolmogorov forward equations to find the marginal distribution,
mean and variance of the number of players in state 0 vs. time.

The other side is to identify the stable fixed points of $\calT.$
Two ways to explore which fixed points of $\calT$ are stable are to either
numerically investigate the orbit trajectories as $\calT$ is repeatedly
applied to some initial trajectory, or to examine the linearization of $\calT$
about a fixed point--this is the Gateaux derivative and it can be expressed
as an integral operator.   The eigenvalues can be computed numerically,
and in rare cases, analytically.   By abuse of notation, we use $\calT$ to
denote the mean field map as a mapping $T(x) \mapsto \tilde x$ obtained
by the change of coordinates $x = 2\theta -1.$ 

\paragraph{Numerical identification of FLMP trajectories}
For the symmetric follow the crowd model,  numerical analysis strongly
and consistently indicates which MFG solutions are FLMP trajectories.
We find that for $\eta \leq 1/2$ they coincide with
the unique MFG equilibrium -- namely, the (0,0) trajectory over $[0,T].$
And for $\eta > 1/2$ there are two FLMP trajectories.   Namely, the one
that traverses the first quadrant in the x-y plane once, and the negative of it,
which traverses the third quadrant  in the x-y plane once.   In particular,
the solutions that wind around the origin through three or more quadrants do not
appear to be FLMP solutions.  See Fig. \ref{fig:bif_curv} for illustration.
\begin{figure}
\centering
\begin{subfigure}{.23\textwidth}\centering
\includegraphics[width=1\linewidth]{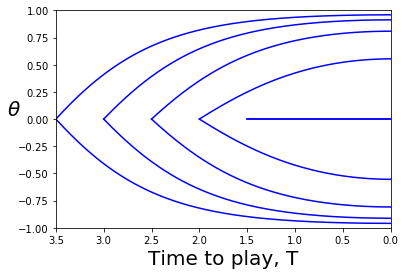}
\end{subfigure}%
\begin{subfigure}{.23\textwidth}\centering
\includegraphics[width=1\linewidth]{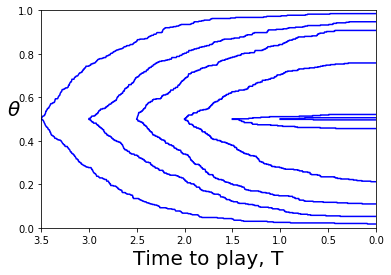}
\end{subfigure}%
\caption{
On the left is a set of realizations of the $N+1$-player game with $400$ players and various time to play,
with initially 200 players in each state.  On the right, are the MFG solutions believed to be
the FLMP trajectories.   Both are for follow-the-crowd game with $\eta=0.$
\label{fig:bif_curv}}
\end{figure}
For less symmetric examples it is less obvious where the {\em bifurcation curve}
is that separates FLMP solutions that converge to a point closer to 1, or converge
to a point closer to 0.   The bifurcation curve often coincides with a line or curve
of indifference for the $N+1$ player game with a large number of players,
corresponding to upcrossings of zero by the mapping $n \mapsto u_1(0,n,t) - u_0(0,n,t) .$
This is illustrated in Fig. \ref{fig:PDE_follow_term}.

\begin{center}
\includegraphics[scale=0.5]{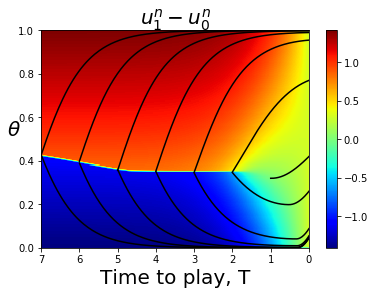} \\
\includegraphics[scale=0.43]{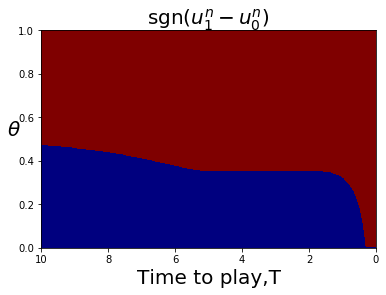}
\includegraphics[scale=0.3]{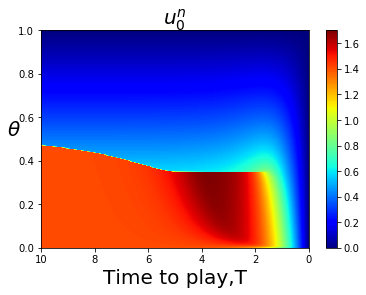}
\includegraphics[scale=0.3]{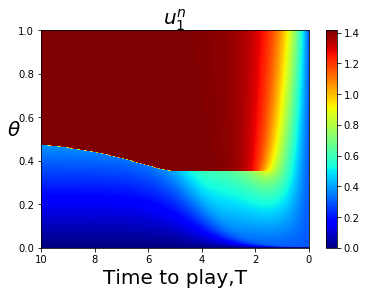}
\captionof{figure}{
Heat maps for cost-to-go functions for follow the crowd, $f(i, \theta)=|\theta-(1-i)|$,
with $N=400$, $T=10$, $\eta=0,$ and asymmetric terminal cost:  $\psi(1)=0.3$ and $\psi(0)=0.$
The MFG equilibrium trajectories beginning at the bifurcation curve are overlaid
onto the heat map of $u_1-u_0$ in the top figure.  
} \label{fig:PDE_follow_term}
\end{center}

\paragraph{Examination of orbits of $\calT$}  
Recall that the fixed points of $\calT$ are the  collective mass trajectories
$(\theta_t : 0\leq t \leq \calT)$ of mean field Nash equilibria.   To numerically
investigate the stability of fixed points of $\calT$ we generated sequences of
iterates of trajectories  $(\theta^n )_{n\geq 0}$ defined by
$\theta^{n+1} = \calT(\theta^n),$  where the initial point $\theta^0$ is a perturbation
of a fixed point.    Figure \ref{fig:more_noise_diverge_nonreal} shows
such sequences of iterates such that the initial trajectory is a
perturbation of one of the two MFG Nash equilibria that cross zero
one time, for the follow the crowd game and time horizon $T=20.$
In both instances, the iterates converged to one of the two equilibria
with no zero crossings.
\begin{figure}\centering
\begin{subfigure}{.25\textwidth}\centering
\includegraphics[width=1\linewidth]{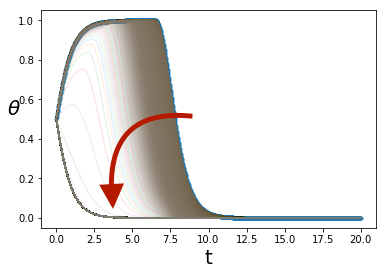}
\end{subfigure}%
\begin{subfigure}{.25\textwidth}\centering
\includegraphics[width=1\linewidth]{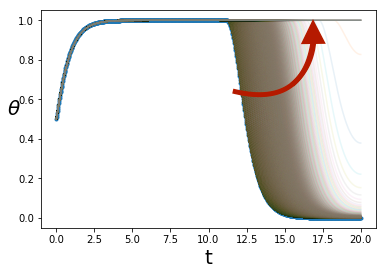}
\end{subfigure}%
\caption{Iterates $(\theta^n )_{0 \leq n \leq 10000}$ for two different initial trajectories
that are perturbations of a single-cross MFG Nash equilibrium, which is indicated by a
thick blue line.} 
\label{fig:more_noise_diverge_nonreal}
\end{figure}

However, overall we found it difficult to numerically verify that a given solution
is not a stable fixed point.  On one hand,  some MFG solutions that we
don't expect to be stable, such as the trajectory that crosses zero once, numerically
appear to be asymptotically stable for a very small basin of stability.   On
the other hand, we have found perturbations of MGF solutions that also
numerically appear to be asymptotically stable, indicating numerical  artifacts are possible.

\paragraph{Linearization of $\calT$ about $(0,0)$}
 Given  a fixed point $\bar x = \calT ( \bar x)$, the Gateaux derivative
$d\calT_X(\bar x, x)$, or the directional derivative of $\cal T$ at $\bar x$
in the direction $x,$  is obtained by linearizing $\cal T$ about $\bar x.$
This is particularly simple if $\bar x$ is the zero trajectory.
\iftoggle{short}
{(Linearization about a nonzero trajectory is given in the full version of this paper.)}  
{(Linearization about a nonzero trajectory is given in Appendix \ref{sec:op_follow}.)}  
In that case, the linearized MFG equations are:
\begin{align}  
\begin{array}{l}
~~ \dot x   = y - 2\eta x  \\
-\dot y = x  - 2\eta y   
\end{array}
\end{align}
Given $(x_u)$,  $\tilde x = d\calT_X(\bar x, x)= \calL_2 \calL_1 x$,
where $\calL_1$ and $\calL_2$ are linear operators defined as:
\begin{align*}
&y_s = (\calL_1 x)_s  =  \int_s^T e^{-2\eta(T-u)}   x_u du    \\
&\tilde x_t  = (\calL_2 y)_t  = \int_0^t  e^{-2\eta(t-s)}  y_s ds
\end{align*}
These expressions can be combined to yield
\begin{align*}
&x_t  = \int_0^T   K(t,u)  x_u du  
\end{align*}
where  $K(t,u) =   e^{-2\eta (t \vee u)} \sinh(2\eta (t \wedge u))/2\eta$ for
$\eta > 0$ and  $K(t,u) =  t \wedge u $ for $\eta =  0.$
In other words,  the Gateaux derivative is the integral operator
with kernel $K.$

If $\eta = 0$, $K(t,u) = t\wedge u,$  which is  the covariance of Brownian
motion, which has a well known Mercer series  expansion.   The eigenvalues
of $K$ are $\lam_n =  \left(\frac{2T}{(2n+1)\pi}\right)^2$ with corresponding
eigenfunctions $h_n(t) = \sin\left(\frac{(2n+1)\pi t}{2T}\right) $ for $n\geq 0.$ 
In particular, the largest eigenvalue is  $\lam_0 =  \left(\frac{2T}{\pi}\right)^2,$
and $\lam_0 \leq 1$ if and only if $T\leq  T_c(0) = \pi/2,$   where $T_c(\eta)$
is the critical time horizon for the appearance of multiple MFG equilibria.

Here is an upper bound on the maximum eigenvalue
of $K$ for $\eta > 0.$   The mappings $\calL_1$ and $\calL_2$ are
both bounded operators in the supremum norm:
$\|y\|_{\infty} \leq c(\eta,T)  \|x\|_{\infty},$
with operator norm $c(\eta,T) = \int_0^T e^{-2\eta t} dt = \frac{1- e^{-2\eta T}}{2\eta}.$
Thus, the Gateaux derivative is also a bounded
operator in the supremum norm with operator bound
$c^2(\eta,T).$\footnote{A somewhat tighter bound is given by
$\|\tilde x\|_{\infty} \leq \tilde c(\eta,T)  \|x\|_{\infty},$
where $\tilde{c}(\eta, T) = \max_t \int_0^T  K(t,s) ds$, but the
expression for $\tilde{c}(\eta, T)$ is complicated.}
Hence, if $\eta \geq 1/2$, the linearized mapping is a contraction
in the $L^{\infty}$ norm for all $T > 0.$   If $\eta < 1/2$ it is a
contraction if $T$ is small enough that $ \frac{1- e^{-2\eta T}}{2\eta} < 1.$

For $\eta > 0$ we conjecture the largest eigenvalue of $K$ is greater than one precisely
when there is a nonzero MFG equilibrium, namely, when
$T > T_c(\eta) \triangleq  \frac{\pi - \arccos(2\eta)}{\sqrt{1-4\eta^2}}.$
We numerically found the largest eigenvalue of the matrix approximation of the
kernel,  $(K(iT/n,jT/n)T/n )_{i, j \in [n]}$ for $n=10^3$
for $\eta \in (0, 0.499)$ and $T$ near $T_c,$ and the calculations match
the conjecture well.

\bibliographystyle{plain}
\bibliography{/Users/brucehajek/Documents/Papers/BIBS/Games_Auctions_ML}

\appendices

\section{Proof of Proposition \ref{prop:FLMP_MFG}}   \label{FLMP_MFG_proof}

This section proves Proposition \ref{prop:FLMP_MFG}, that
if $\eta > 0,$  FLMP trajectories are mean field game equilibria.
The proof is given after some initial notation is given and two lemmas are proved.
Let  $(\theta_t)_{0\leq t \leq T}$ be an FLMP trajectory and let $(i^N(0), n^N(0))_{N\geq 1}$ be a corresponding
sequence of initial conditions as in the definition of FLMP trajectory. 
For $N\geq 1$,   let $( (i(t), n(t)) : 0 \leq t \leq T)$ denote the controlled Markov process for $N+1$
players resulting for initial state $(i^N(0), n^N(0)),$ when all players use the unique policy $(\alpha^*(i,n,t))$
for the Markov perfect equilibrium for $N+1$ players.  
Since the functions $f(i,\theta, t)$ and $\psi(i,\theta)$ are bounded,   for $T$ fixed, the cost to go functions $u(i,n,t)$ determined
by the HJB equations  \eqref{eq:HJB_N+1}- \eqref{eq:HJB_N+1_bc} are uniformly bounded
for all $N, i, n,$ and $t\in [0,T].$     Therefore, the policy $\alpha^*$, determined by \eqref{eq:HJB_N+1_policy},
is also uniformly bounded.     Select $\Gamma_1$ such that  $(\alpha^*(i,n,t)) \leq \Gamma_1$ for all $N, i,n,$ and $t\in [0,T].$
Suppose also that $\Gamma_1$ is large enough that $\alpha(i,t) \leq \Gamma_1$ for all $i, t$ for any
decentralized policy $\alpha(i,t)$ resulting by responding to a deterministic collective mass trajectory.

Consider the following variation of the Markov perfect equilibrium.   Suppose
the reference player switches from using $\alpha^*$ to some other policy, $\beta^*(i,n,t),$
such that $\beta^*(i,n,t) \leq \Gamma_1$ and $t \mapsto \beta^*(i,n,t)$ is continuous for all $(i,n).$
Let $P$ denote the original probability distribution
for the process $(i(t),n(t))_{0\leq t \leq T}$ and let $\tilde P$ denote  the probability distribution of
$(i(t),n(t))_{0\leq t \leq T}$  when the reference player switches to policy $\beta^*.$

\begin{lemma} (Insensitivity of FLMP trajectory to one player switching policies)  \label{lmm:insensitivity}
The following holds for any $\eps>0$, 
\begin{align} \label{eq:perturbed_fluid_lim}
\lim_{N \to \infty}  \tilde\P\left[   \bigg|  \frac{n^N(t)}{N+1} - \theta_t \bigg|  < \eps \mbox{ for } 0 \leq t \leq T  \right] =1. 
\end{align}
\end{lemma}

\begin{lemma}  \label{lmm:PP}
Let $P$ and $\tilde P$ be probability distributions on the same measurable space $(\Omega, \calF)$
such that $\tilde P <<P$  (i.e. $\tilde P$ is absolutely continuous with respect to $P$) and let $\frac{d\tilde P}{dP}$
denote the Radon-Nikodym derivative.   Suppose
$E_P\left[   \left( \frac{d\tilde P}{dP}   \right)^p     \right]^{1/p}  \leq c$ for some $p > 1$ and $c.$
Let $q>1$ be such that $\frac 1 p + \frac 1 q = 1.$
Then for any event $A$,   $\tilde P(A) \leq  c P(A)^{1/q}.$
\end{lemma}
\begin{proof}[Proof of Lemma \ref{lmm:PP}]
By H\"{o}lder's inequality,
\begin{align*}
\tilde P(A)  & = \int_{\Omega}  \frac{d\tilde P}{dP} \identityf{A}   dP   \\
 & \leq  c  \left( \int_{\Omega} \identityf{A}^q dP  \right)^{1/q} = c P(A)^{1/q}
\end{align*}
\end{proof}

\begin{proof}[Proof of Lemma \ref{lmm:insensitivity} ]
Since $P$ and $\tilde P$ only differ by the change in the policy for player 1,  the Radon-Nikodym derivative $\frac{d\tilde P}{dP}$
can be written explicitly as follows.   Let $(Y_t)_{0\leq t \leq T}$ denote the number of jumps of the state of the reference
player during $[0,t].$   Then, by standard theory of  change of probability measure  for point processes (Girsanov type result
for point processes, see \cite{VanSchuppenWong74}, Theorem 4.1 for example),  $\tilde P << P$ and the
Radon-Nikodym derivative is given by
\begin{align*}
\frac{d\tilde P}{dP} = \exp\left( \int_0^T  \ln\left(\frac{\beta^*+ \eta}{\alpha^* + \eta} \right) dY_t   - \int_0^T  (\beta^*-\alpha^*)  dt \right)
\end{align*}
where $\beta^*$ is short for $\beta^*(i(t-),n(t-),t),$ $\alpha^*$ is short for  $\alpha^*(i(t-),n(-),t)$ and $\eta$ is
the fixed positive background jump rate.

Note that for $p > 1$,  the expression for  the Radon-Nikodym derivative to the $p^{th}$ power can be written
as a product
{\small
\begin{align*}
\left(  \frac{d\tilde P}{dP}\right)^p  
=  \frac{ d\tilde{\tilde P} }{dP}  \eexp ^{  \int_0^T  ( \beta^*+ \eta)^p - ( \alpha^* + \eta )^p - p(\beta^* - \alpha^*) dt} 
\leq \frac{ d\tilde{\tilde P} }{dP} \Gamma_2
\end{align*}
}where $\tilde{\tilde P}$ is a probability measure corresponding to a similar Radon-Nikodym derivative
with a factor $p$ in front of the log term, and
 $\Gamma_2 = \exp\left[ T \left(   (\Gamma_1 + \eta)^p + p\Gamma_1   \right)  \right].$
 Thus,   $E_P\left[  \left(  \frac{d\tilde P}{dP}\right)^p   \right] \leq \Gamma_2.$
 Lemma  \ref{lmm:insensitivity}  thus follows from Lemma \ref{lmm:PP} with $A$ equal to the
complement of the event in \eqref{eq:perturbed_fluid_lim}.
 \end{proof}

 \begin{proof}[Proof of Proposition \ref{prop:FLMP_MFG}]
  Consider the Markov perfect equilibrium for large $N.$ 
 In view of Lemma   \ref{lmm:insensitivity}, if the reference player
 deviates from using $\alpha^*$, the normalized process $n(t)/N$
 for the rest of the population still follows $\theta$ arbitrarily
 closely as $n\to\infty.$   Thus, an asymptotically optimal policy
 for the reference player to switch to is the optimal response to
deterministic collective mass trajectory $\theta.$  Furthermore, it
implies $u(n,i,t) - u(i,t)$ converges  to zero uniformly in $n$ and
$t \in [0,T]$,  where $u(n,i,t)$ is associated with the $N+1$ player
MP equilibrium, and $u(i,t)$ is the cost-to-go for the single reference
player responding to the deterministic mass trajectory  $\theta.$
It follows that all players in the  $N+1$ game are asymptotically
effectively using the same policy as the  alternate policy of the reference player.
(in other words,  $ ( u(i,n,t)-u(1-i,n,t) )_+ \approx ( u(i,t)-u(1-i,t) )_+ ).$
Thus, the corresponding  fluid limit is the same as the mean limit
for the reference player  with random initial state equal to 0
with probability  $n^M(0)/n.$
 \end{proof}

\iftoggle{short}{}{
 \section{The uniform law of large numbers}  \label{sec:unif_law_large_numbers}
 
 Theorem 7.4 of  \cite{GineZinn} is repeated here for convenience.
 
\begin{proposition}
 Let $(X_t : 0\leq t \leq T)$ be a centered, stochastically continuous
 uniformly bounded random process whose trajectories are right
 continuous and have left limits.   Assume for some $c> 0$, some
 nondecreasing function $F \in D[0,1]$ and for all $s,t \in [0,1],$
 $E[ |X_t -X_s| \leq |F(t) -F(s)|.$   Then $X \in CLT$  in $(D([0,1], \| \cdot  \|_\infty).$
 \end{proposition}
 
 An implication of this theorem is that if all players use the
 same decentralized policy $\alpha(i,t)$ (assumed to be bounded
 and measurable in $t$) and if the initial conditions satisfy $\frac{n(0)}{N} \to \overline \theta$
 for some $\overline \theta  \in [0,1],$  then as $n\to\infty,$   the population
 average converges to $(\theta_t)$ in probability in the supremum norm, where $(\theta_t)$ is determined
 by the Kolmogorov forward equation
  \begin{align*}
&\dot{\theta_t} = (1-\theta_t)( \alpha(0,t) + \eta)  -   \theta_t  (\alpha_1(t)  + \eta )   \\
&   \theta_0 = \overline{\theta}  ~~~~~~~~~~~~~~\mbox{boundary condition at 0} 
\end{align*}
 
 Therefore, the following are equivalent for a trajectory $(\theta_t)$:
 \begin{enumerate}[(a)]
\item Let $\alpha^*$ denote the optimal response policy for a single player
 in response to $\theta.$  In other words, $\alpha(i,t) = (u(i,t) - u(1-i,t))_+$ where $(u(i,t))$ is determined
 by  \eqref{eq:HJB_one_player}- \eqref{eq:HJB_one_player_init}.  Then for any $\epsilon > 0$ and any sequence of finite
 player games with $\frac{n(0)}{N} \to \theta_0$,  
  the strategy profile of all players using $\alpha^*$ is an $\epsilon$-Nash equilibrium
  for sufficiently large $N$. 
  \item $(\theta_t)$ is the population trajectory of a MFG equilibrium.
  \end{enumerate}
  
  \section{Monotonicity of period with amplitude}  \label{sec:period_monotonicity}
  
 Consider the follow the crowd dynamics \eqref{eq:follow_MFGxy}, rewritten here for convenience:
\begin{align*} 
\begin{array}{l}
 \dot x   = y-x|y| - 2\eta x  \\
\dot y =  - x + \frac12 y|y| + 2\eta y   
\end{array}
\end{align*}
  
 From \eqref{eq:follow_MFGxy} we find for all $x,y,$
\begin{align*}
\ddot y &=  -\dot x  +  |y| \dot y  + 2\eta  \dot y \\
  &= \frac 1 2 y^3 + 3\eta y|y| + (4\eta^2 - 1) y.
\end{align*}
Equivalently, writing $v= \dot y$, yields
\begin{align}   \label{eq:follow_yv}
\begin{array}{l}
\dot y   = v   \\
\dot v   = \frac 1 2 y^3  + 3\eta y|y| + (4\eta^2 - 1) y.
\end{array}
\end{align} 
The motion \eqref{eq:follow_yv} admits the Hamiltonian $H(y,v) = \frac 1 2 v^2 - \frac 1 8 y^4  - \eta |y|^3  - \frac{4\eta^2-1}{2} y^2.$
If $\eta < 1/2$ then $H(y,v)$ is convex near the origin.  Letting $\varphi = \arctan\left( \frac{y}{v} \right)$ we find
\begin{align*}
\dot \varphi (y,v) =  \frac{\dot v  y - v \dot y}{y^2 + v^2}  =   - 1  +  \frac{ 4\eta^2 y^2  + \frac 1 2 y^4  + 3 \eta |y|^3 }{y^2 + v^2}
\end{align*}
Note that $\dot \varphi$ is increasing in $|y|$  for any fixed ratio of $v$ to $y$  (decreasing angular speed).
Hence, the period of the periodic trajectories increases with amplitude.

\section{Linear asymptotic stability for symmetric follow the crowd example}  \label{app:linear asymptotic stability}

A definition of {\em linear asymptotic stability} was introduced in (\cite{YinMehtaMeynShanbhag12}, Section IV) for a constant in time
solution to the infinite horizon long term average cost mean field game.   We translate that definition to our setting.
Roughly speaking,  linear asymptotic stability is a variation, based on linearization, of the asymptotic stability properties
delineated in Section  \ref{sec:infinite_horizon_transient}.
\begin{definition}
Suppose $(\hat x, \hat y )$ is an equilibrium point for the ode  \eqref{eq:follow_MFGxy}.
Seeking solutions of the form $(\tilde x , \tilde y) =  (\hat x, \hat y ) + \epsilon (x,y) +  o(\epsilon)$,
we obtain a linear initial value problem for $(x,y)$ by linearizing \eqref{eq:follow_MFGxy} about $(\hat x, \hat y ).$
The point $(\hat x, \hat y )$ is said to be {\em linearly asymptotically stable} if for any initial perturbation $x_0 \in \reals,$   there exists
a unique solution $(x_t, y_t)_{t\geq 0}$ to the linearized equations  (with the given initial condition for $x$, some initial condition for $y$,
and satisfying the $L^2$ constraint $\int_0^{\infty} \|x_s - \hat x\|^2 ds < \infty$) and, furthermore,  $\lim_{t\to\infty} x_t = \hat x.$
\end{definition}

With the definition in place we prove the following proposition.
\begin{proposition}  \label{prop:lin_asymp_stab_00}
The equilibrium point $(0,0)$ is linearly asymptotically stable if and only if $\eta > \eta_c = 1/2.$
The equilibrium points $\pm \hat P$ are linearly asymptotically stable if and only if $0\leq \eta < 1/2.$
\end{proposition}

\begin{proof}   For an equilibrium point $(\hat x, \hat y)$, we have $H_x(\hat x, \hat y) =H_y(\hat x, \hat y) = 0$ and
\begin{align*}
H_y(\hat x + \epsilon x,\hat y + \epsilon y)   =  \epsilon H_{xy}(\hat x, \hat y) x  + \epsilon H_{yy}(\hat x, \hat y) y + o(\epsilon)  \\
H_x(\hat x + \epsilon x,\hat y + \epsilon y)   =  \epsilon H_{xx}(\hat x, \hat y) x  + \epsilon H_{xy}(\hat x, \hat y) y + o(\epsilon).  \\
\end{align*}
So the linear initial value problem for $(x,y)$ can be written as
\begin{align}  \label{eq:xdotAx}
\binom{\dot x}{\dot y} = A \binom{x}{y}
\end{align}
where
\begin{align*}
 A = \left( \begin{array}{cc} H_{xy} &  H_{yy} \\  -H_{xx} & -H_{xy} \end{array} \right)\bigg|_{\hat x, \hat y}  
\end{align*}
Since $\Tr(A)=0$ (so sum of eigenvalues is zero) and $\det(A) = \det( \calH (\hat x, \hat y))$ where $\calH$ is the Hessian of $H$:
\begin{align*}
\calH(\hat x, \hat y) = \left( \begin{array}{cc} H_{xx} &  H_{xy} \\  H_{xy} &  H_{yy} \end{array} \right)\bigg|_{\hat x, \hat y},
\end{align*}
the eigenvalues of $A$ are $\pm \sqrt{ - \det(\calH(\hat x, \hat y) )}.$  If $\det(\calH)  < 0$ then the eigenvalues are real valued and one is negative.
If $\det(\calH)  > 0$ the eigenvalues are purely imaginary.  
For the follow the crowd game with Hamiltonian given in \ref{eq:Hxy},
\begin{align*}
 A  = \left( \begin{array}{cc} -2\eta - |\hat y| &  1  \\  -1 &   2\eta + |\hat y|   \end{array} \right).
\end{align*}

Consider first the zero equilibrium, $(\hat x, \hat y) = (0,0),$  in which case  $A =\left( \begin{array}{cc}  - 2\eta    &   1 \\ -1 & 2\eta \end{array}\right) .$
This $A$ has eigenvalues $\pm \sqrt{4\eta^2 -1}$ and, for $\eta \geq 0.5$, corresponding eigenvectors $\binom{1}{2\eta \pm \sqrt{4\eta^2 -1}}.$
If $\eta > 0.5$ then the solutions to \eqref{eq:xdotAx} have the following form, for some constants $a$ and $b$,
\begin{align*}
&a \binom{1}{2\eta + \sqrt{4\eta^2 -1}}   \eexp^{ t \sqrt{4\eta^2 -1}} \\
&~~~~\qquad +  b \binom{1}{2\eta - \sqrt{4\eta^2 -1}}   \eexp^{- t \sqrt{4\eta^2 -1}}
\end{align*}
The initial condition for $x$ and the $L^2$ constraint are satisfied if and only if  $a=0$ and $b=x_0,$   and the resulting solution converges
to zero as $t \to \infty.$  Hence, the system is linearly asymptotically stable if $\eta > 1/2.$   If $\eta < 1/2$ then the two eigenvalues of $A$ are
purely imaginary, nonzero, and negatives of each other, so that all nonzero solutions to \eqref{eq:xdotAx} are periodic.    If $\eta = 1/2$ all solutions have
$x$ of the form $x_t = a + bt.$   So, combining the observations for $\eta < 1/2$  and $\eta=0.5$, we conclude that for $\eta \leq 1/2$ there are no
nonzero solutions satisfying the $L^2$ constraint.   So for $\eta \geq 1/2$ the zero equilibrium is not linearly asymptotically stable.

Now consider the equilibrium point $\pm \hat P$ and suppose $\eta < 1/2.$   Then, since $| \hat y | = \sqrt{2+\eta^2} - 3 \eta,$  we
find that $2\eta + |\hat y| = \sqrt{2+\eta^2} - \eta  > 1.$   Therefore, by the analysis for the zero equilibrium point with $2\eta$ replaced
by $2\eta + |y|$, we see that again $A$ has two real-valued eigenvalues of opposite sign, so the system is linearly asymptotically stable.
\end{proof}

Note that the eigenvalues $\pm \sqrt{4\eta^2 -1}$ for the equilibrium point $(0,0)$
have qualitatively the same graph as in Figure 2(b) of \cite{YinMehtaMeynShanbhag12},
with $R$ and $R_c$  replaced by $\eta$ and $\eta_c.$

\begin{remark}
Proposition \ref{prop:lin_asymp_stab_00} illustrates the notion of linear asymptotic stability for equilibrium points
of the infinite horizon average cost mean field game, introduced in \cite{YinMehtaMeynShanbhag12}.     The two state
Markov control problem we have considered is considerably simpler than the coupled oscillator problem considered
in  \cite{YinMehtaMeynShanbhag12}, so, as explained in Section \ref{sec:infinite_horizon_transient},
we could observe asymptotic stability properties of equilibrium points directly,
rather than considering the linearized dynamics.
\end{remark}

%
%



\section{Kernel for Gateaux derivative of $\calT$ for nonzero $x.$} \label{sec:op_follow}

We give an expression for the kernel of the Gateaux derivative along a nonzero $x$ trajectory in
case $\eta=0$  for follow the crowd cost function. 
Given $x$,  $\tilde x = T(x)$ is found by first finding $y$:
\begin{align} \left\{
\begin{aligned} \label{eq:y_op1}
-\dot y &= x - \frac12 y |y| \\
y_T &= 0, 
\end{aligned} \right.  
\end{align}
and then $\tilde x:$ 
\begin{align} \left\{
\begin{aligned} \label{eq:x_op2}
\dot {\tilde x} &= y-\tilde x |y| \\ 
\tilde x_0 &= x_0. 
\end{aligned} \right.
\end{align} 
Fix $\hat t \in (0,T)$ and $\eps > 0$ sufficiently small. 
Suppose $h(t) = \delta(t-\hat t)$. 
Let $x_\epsilon = x+\epsilon h$, $y_\epsilon = y + \epsilon  k+ o(\epsilon)$
and $\tilde x_\epsilon = x + \epsilon g + o(\epsilon).$
Let $y, y_\eps$ be the solution to \eqref{eq:y_op1} with the $x, x_\eps: [0,T] \to [-1,1]$ respectively. 
Then, linearizing the equations for $y$and $\tilde x$ yields 
\begin{align*}
- \dot k &= h  -  k |y|, \text{ with } k(T)=0  \\
\dot g & = k - |y| g  -  x \sgn(y)  k, \text{ with } g(0) =0 
\end{align*} 
so that
\begin{align*}
k(s) &= \int_s^T    e^{-\int_s^{\hat t} |y| dr} h(\hat t) d\hat t \\
g(t) &= \int_0^t  e^{-\int_s^t |y_r| dr}k(s) (1-\sgn(y_s)x_s) ds,
\end{align*} 
and the kernel of $\calT$ is thus given by
\begin{align*}
K(t,\hat t)= \int_0^{t \wedge \hat t} e^{-\int_s^t |y| dr} e^{-\int_s^{\hat t} |y| dr} (1-\sgn(y_s)x_s)ds. 
\end{align*}
If $y\geq 0$ over $[0,T]$ then $(1-\sgn(y_t) x_t) = e^{-\int_0^t |y| dr},$ yielding:
\begin{align*}
K(t,\hat t)=\int_0^{t \wedge \hat t} e^{-\int_s^t |y| dr} e^{-\int_s^{\hat t} |y| dr} e^{-\int_0^s |y| dr} ds.
\end{align*}

\section{Avoid the crowd cost function}

In contrast to the follow the crowd game focused on in this
paper, the MFG equilibrium for the avoid the crowd game of
this section has a unique solution.   Suppose the cost per time spent in state $i$ is
$$ f(i,\theta)=|i -\theta  | =\left\{ \begin{array}{cl} \theta & i=0 \\
1-\theta & i=1
   \\
\end{array} \right.
$$
where $\theta$ is the fraction of other players in state 0. \\
The reduced dimension  MFG equations become
\begin{align}   \label{eq:avoid_MFGxy}
\begin{array}{l}
~~ \dot x   = y-x|y| - 2\eta x  \\
-\dot y = -x - \frac12 y|y| - 2\eta y   
\end{array}
\end{align}
with associated Hamiltonian
\begin{align} \label{eq:_avoid_Hxy}
H(x,y)= \frac{-x^2 - 4 \eta xy + y^2 - xy|y|}{2}.
\end{align}
Contour maps of $H$ are shown in  Fig.  \ref{fig:avoid_beta_Ham} for two values of $\eta.$
\begin{figure}[htbp]
\post{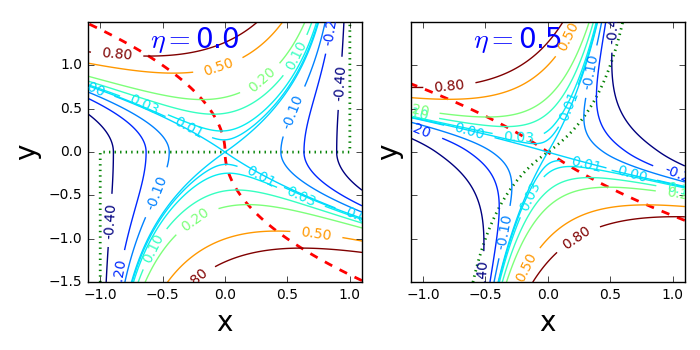}{8}
\caption{Avoid the crowd case.   Contour plot of $H$ for several values of $\eta.$   Dashed lines are the zero sets
of  $H_x,$ and dotted lines are the zero sets of $H_y.$  The intersections
of dotted and dashed lines are the critical points of $H$  (i.e. solutions to $\nabla H = 0.$)}
\label{fig:avoid_beta_Ham}
\end{figure}
We observe that $(0,0)$ is the unique critical point of $H,$ and
for any $x_0\in [-1,1]$ there exists a unique value of $y_0$ such that the solution of
the initial value problem with dynamics \eqref{eq:avoid_MFGxy} over $[0,\infty)$ is such that
$x_t \in [-1,1]$ for all $t.$    Furthermore, such solution converges to $(0,0)$. 
Also,  $\det  \mbox{Hess} H (0,0)  = -1 - 4\eta^2 < 0,$ and the unique equilibrium point $(0,0)$
of the infinite horizon average cost MFG is linearly asymptotically stable.

\section{On the difference of cost to go for $N+1$ players}   \label{sec:difference_N}

Recall that using  $y_t$ defined by  $y_t = u(1,t) -u(0,t)$ yielded a reduction from three to two dimensions
in the MFG equilibrium equations.   Let us see if a similar reduction occurs for the
Nash equilibrium equations for the $N+1$ player game.
For convenience we restate the HJB cost-to-go equations \eqref{eq:HJB_N+1} and \eqref{eq:HJB_N+1_policy}
for the reference player in the $N+1$ player game:
\begin{align}
&-\dot{u}(i,n,t)   =  f(i,n) -  \frac{((\alpha^*(i,n,t))^2}{2}\nonumber   \\
&~~~~ + \eta (u(1-i,n,t) - u(i,n,t))  \nonumber \\
& ~~~~ +  \gamma^+(i,n,t)(u(i,n+1,t)-u(i,n,t))  \nonumber  \\
&~~~~+ \gamma^- (i,n,t)(u(i,n-1,t)-u(i,n,t)),  \label{eq:HJB_N+1_repeated} \\
& u(i,n,T)=\psi(i,n),   \label{eq:HJB_N+1_bc}
\end{align}
where the corresponding control policy is
\begin{align}
\alpha^*(i,n,t)= ( u(i,n,t)-u(1-i,n,t) )_+  \label{eq:HJB_N+1_policy_repeated}.
\end{align}
Suppose all players use policy $\alpha^*$,  so $\beta = \alpha^*$ in the definition of $\gamma^{\pm}.$
Let $Y(n,t) = u(1,n,t) - u(0, n, t),$ $\delta f(n) =  f(1,n) - f(0,n),$
and $\delta \psi(n) = \psi(1,n)  -  \psi(0,n) .$    Using the facts
\begin{align*}
\alpha^*(1,n,t) &= (Y(n,t))_+  \\
\alpha^*(0,n,t) &= (-Y(n,t))_+  \\
 \gamma^+(i,n,t)&=(N-n) (\alpha^*(1, n+1-i,t) + \eta)    \\
&= (N-n)  (Y(n+1-i,t)_+ + \eta )   \\
 \gamma^-(i,n,t)&=  n (\alpha^*(0, n-i,t) + \eta)   \\
 &=  n (  (-Y(n-i,t))_+  + \eta)
\end{align*}
in \eqref{eq:HJB_N+1_repeated} yields
\begin{align*}
&-\dot Y(n,t)   \\
&= \delta f(n)  -  \frac{\left|Y(n,t)\right|  Y(n,t)}{2}   - 2\eta Y(n,t)   \\
& + (N-n) \eta (Y(n+1,t)-   Y(n,t))  \\
&  +  n\eta ( Y(n-1,t) - Y(n,t))  \\
& +  (N-n)  Y(n,t)_+  (u(1,n+1,t)-u(1,n,t))    \\
&+  n  (-Y(n-1,t))_+ (u(1,n-1,t)-u(1,n,t)),  \\
&  -  (N-n)  Y(n+1,t)_+ (u(0,n+1,t)-u(0,n,t))    \\
&-  n   (-Y(n,t))_+  (u(0,n-1,t)-u(0,n,t)), \\
& Y(n,T)=\delta \psi(n)
\end{align*}

The RHS is not a function of $Y$ alone.   However,
using $Y(n,t)_+ \approx Y(n+1,t)_+$ and  $(-Y(n-1,t))_+ \approx  (-Y(n,t))_+ $
yields the approximation:
\begin{align}
&-\dot Y(n,t)   \approx \delta f(n)  -  \frac{\left|(Y(n,t)\right|  Y(n,t)}{2}  - 2\eta Y(n,t)  \qquad \nonumber \\
& + (N-n)(Y(n,t)_+  + \eta)(Y(n+1,t)-   Y(n,t))  \nonumber \\
&+  n ( (-Y(n,t))_+ + \eta) (Y(n-1,t)-Y(n,t))  .   \label{eq:Y_approx}  \\
&Y(n,T)=  \delta \psi(n)   \label{eq:Y_approx_boundary}
\end{align}

\section{The MFG partial differential equation}

An interpretation of a mean field game Nash equilibrium $(u(i,t) , \theta_t)$ is that at each time $t$,
$u(i,t)$ is the cost to go for a reference player in state $i$, given that the fraction of players
in state 0 is $\theta_t.$   That picture can be embedded into a larger picture.   Bt taking a limit
of the HJB equations for $N+1$ players as $N\to\infty,$   we can derive a PDE for $(U(i,\theta, t))$
such that $U(i,\theta, t)$ is the cost-to-go for a reference player in state $i$ given that the fraction
of players in state 0 is  $\theta$   for any $\theta \in [0,1].$       This idea is described in
\cite{gomes2013continuous} (see Proposition 8) and is attributed there to P. Lions.
For simplicity, we derive the PDE for the avoid the crowd game and use the equations
derived in Section \ref{sec:difference_N}.   We use notation $Y$ instead of $U$ and $x$ instead of $\theta.$

Equations  \eqref{eq:Y_approx}-\eqref{eq:Y_approx_boundary}
suggest the following PDE, where now $n$ is treated as a continuous variable over $[0,N]$
rather than as an integer variable.
\begin{align}
-\frac{\partial Y}{\partial t}  &=  \delta f  -  \frac{\left|Y\right|  Y}{2} - 2\eta Y \nonumber  \\
& + \left[ (N-n)(Y_+  + \eta) -  n ( (-Y)_+ + \eta) \right] \frac{\partial Y}{\partial n}.   \label{eq:PDE_Y}  \\
&Y_T=\delta \psi  \label{eq:PDE_Y_boundary}
\end{align}

Note that if we let $(\hat n_t, 0 \leq t \leq T)$ be defined by the following initial value problem
\begin{align}   \label{eq:hat_n_IVP}
\dot {  \hat n} =    (N-\hat n)(Y(\hat n_t))_+    + \eta) -  \hat n ( (-Y(\hat n , t))_+ +  \eta)
\end{align}
then by the chain rule and the PDE \eqref{eq:PDE_Y},
\begin{align}  
-\frac{ Y(\hat n, t)}{d t}  & =  \delta f(\hat n)    -\frac{\partial Y}{\partial t}(\hat n, t)   -\frac{\partial Y}{\partial n}(\hat n, t)   \dot {  \hat n} \nonumber   \\
 &=  \delta f(\hat n)  -  \frac{\left|Y(\hat n, t)\right|  Y(\hat n, t)}{2} - 2\eta Y(\hat n, t) .  \label{eq:Yhatn_TVP}    \\
Y(\hat n_T, T) &=\delta \psi(\hat n_T)
\end{align}
Note that if we set $y_t =  Y(\hat n, t)$ and $x_t = \frac{2\hat n_t} N -1$ and consider
the join-the-crowd cost function (so  $f(\hat n)=\frac{2n-N} N$),    then
 \eqref{eq:hat_n_IVP} and \eqref{eq:Yhatn_TVP} are equivalent to the  MFG equations  \eqref{eq:follow_MFGxy}.
 This calculation is an instance of Proposition 8 of \cite{gomes2013continuous}.
  Figure 8 gives numerical evidence that $u(1,n,t)-u(0,n,t)$, with $n$ normalized to $\theta \in [0,1]$, converges
 as $n\to\infty.$   Presumably the limit is the solution of the PDE. 
 
 The PDE \eqref{eq:PDE_Y}- \eqref{eq:PDE_Y_boundary} is a first order hyperbolic type.
 Equation \eqref{eq:hat_n_IVP} defines a characteristic curve for the PDE, which is why the PDE along
the  curve reduces to an ODE.   The fact there are multiple MFG solutions indicates that solutions of the
PDE are also not unique.   The problem of identifying which MFG Nash  equilibria are FLMP trajectories
therefore can be extended to the problem of determining which solutions of the PDE are limits of the scaled
cost-to-go functions $((u(i, n, t)).$ 
\begin{center}
\includegraphics[scale=0.5]{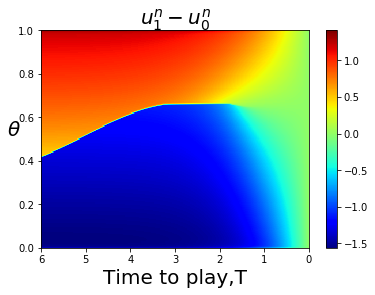}
\includegraphics[scale=0.4]{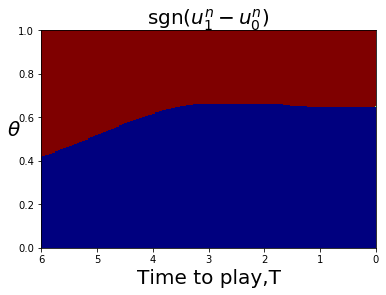}
\includegraphics[scale=0.3]{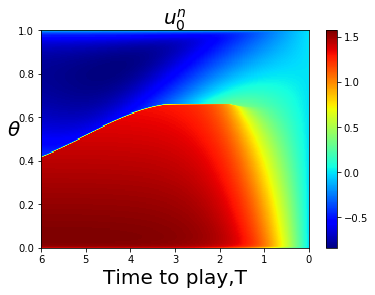}
\includegraphics[scale=0.3]{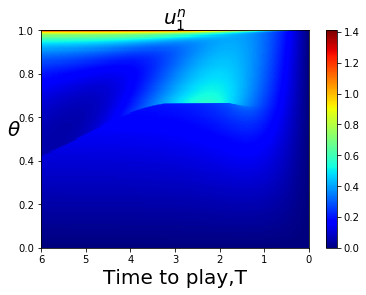}
\captionof{figure}{
Heat maps of cost to go functions for $N=400$ for an example with follow the crowd tendency with a prisoners' dilemma cost added in.
Running cost has $f(i, \theta)=|1-i-\theta|-0.6 \theta +0.3\ind_{i=0}$ and terminal cost is zero. 
State $0$ is the cooperative state and state $1$ is the greedy state. 
The join-the-crowd social pressure cost is given by $|1-i-\theta|$, the cooperative cost is given by $0.6 \theta$, and
the individual incentive cost is given by $0.3 \ind_{i=0}$.    All players would be better off if they moved to state 0, but if
$\theta$ is near $1/2$ then players have incentive to move to state 1.} \label{fig:PDE_follow_Dilemma}
\end{center}

\begin{center}
\includegraphics[scale=0.28]{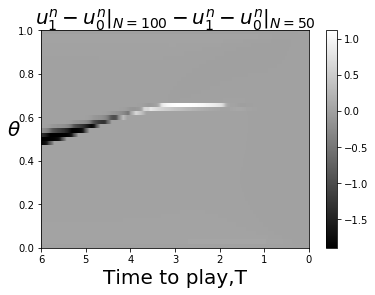}
\includegraphics[scale=0.28]{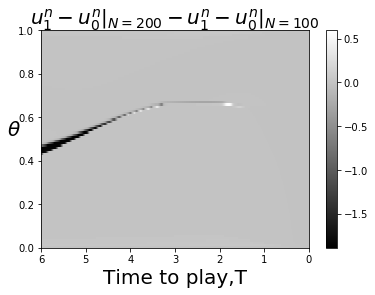} 
\includegraphics[scale=0.28]{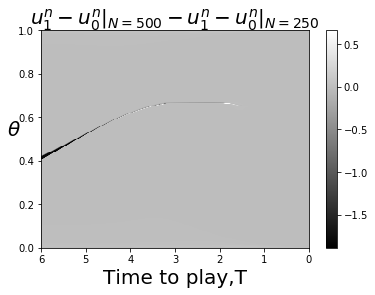}
\includegraphics[scale=0.28]{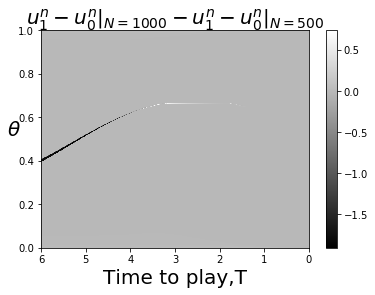}
\captionof{figure}{
Illustration of the pointwise convergence of the indifference set shown in Fig. \ref{fig:PDE_follow_Dilemma}
as $N$ increases.   Heatmaps of differences of  $u_1-u_0$  for different values of $N$ are shown, specifically, for
$N$ values: 100 vs. 50,  200 vs. 100, 500 vs. 250,  and 1000 vs. 500.} \label{fig:indiff_diff_dilemma}
\end{center}
}

\end{document}